\documentclass[12pt, onesided]{amsart}

\usepackage{amsmath}
\usepackage{amssymb}
\usepackage{amsthm}
\usepackage[all,cmtip]{xy}

\theoremstyle{plain}
\newtheorem{thm}{Theorem}[section]

\newtheorem{lem}[thm]{Lemma}

\newtheorem{prob}[thm]{Problem}

\theoremstyle{definition}

\newtheorem{eg}[thm]{Example}

\theoremstyle{remark}
\newtheorem{rem}[thm]{Remark}
\newtheorem{claim}[thm]{Claim}

\DeclareMathOperator{\Def}{Def}

\DeclareMathOperator{\Spec}{Spec}

\begin{document}

\title
[Deformations of weak Fano manifolds ]
{ Unobstructedness of deformations \\ of weak Fano manifolds}

\author{Taro Sano}
\address{Mathematics Institute, Zeeman Building, University of Warwick, Coventry, CV4 7AL, UK}
\email{T.Sano@warwick.ac.uk}

\dedicatory{Dedicated to Professor~Miles~Reid on the~occasion of his~65th~birthday.}

\maketitle

\begin{abstract}
We prove that a weak Fano manifold has unobstructed deformations. 
For a general variety, we investigate conditions under which a variety is necessarily obstructed. 
\end{abstract}

\tableofcontents

 \section{Introduction}
 We consider algebraic varieties over an algebraically closed field $k$ of characteristic zero. 

The Kuranishi space of a smooth projective variety has bad singularities in general.  
Even in the surface case, Vakil \cite{vakil} exhibited several examples of smooth projective surfaces of general type with arbitrarily singular Kuranishi spaces. 

On the other hand, in some nice situations, the Kuranishi space is smooth. 
A famous result is that the Kuranishi space of a Calabi-Yau manifold is smooth. 
The Kuranishi space of a Fano manifold $X$ is also smooth since $H^2(X, \Theta_X)=0$ by the Kodaira--Nakano vanishing theorem, 
where $\Theta_X$ is the tangent sheaf of $X$. 

In this paper, we look for several nice projective manifolds with smooth Kuranishi space.

A smooth projective variety $X$ is called a {\it weak Fano manifold} if the anticanonical divisor $-K_X$ is nef and big.  
The following is the main theorem of this paper. 

\begin{thm}\label{wfanothm}
 Deformations of a weak Fano manifold are unobstructed. 
\end{thm}

Previously, Ran  proved the unobstructedness for a weak Fano manifold with a smooth anticanonical element (\cite[Corollary 3]{ran}). 
Minagawa's argument in \cite{minagawaosaka}  implies the unobstructedness when $|{-}2K_X|$ contains a smooth element. However 
these assumptions are not satisfied for a general weak Fano manifold as explained in Example \ref{egelep}. 
We prove it for the general case. 

We use the $T^1$-lifting technique developed by Ran, Kawamata, Deligne and Fantechi-Manetti. 
Another approach is dealt with by Buchweitz--Flenner in \cite{BF}. 

The following more general result implies Theorem \ref{wfanothm}. 
\begin{thm}\label{main}
Let $X$ be a smooth projective variety. Assume that $H^1(X, \mathcal{O}_X)=0$ and there exists a positive integer $m$ and a smooth divisor 
$D \in |{-}m K_X|$ such that $H^1(D, \mathcal{N}_{D/X})=0$. 

Then deformations of $X$ are unobstructed. 
\end{thm}

We sketch the proof of Theorem \ref{main}. 
Instead of proving the unobstructedness directly, we first prove the unobstructedness for the pair of a weak Fano manifold $X$ 
and a smooth element $D$ of $|{-}mK_X|$ 
for a sufficiently large integer $m$ in Theorem \ref{main1}. Next we show that the unobstructedness for $(X,D)$ implies the unobstructedness for $X$. 

We also show that the Kuranishi space of a smooth projective surface is smooth if the Kodaira dimension of the surface is negative or $0$ in Theorem \ref{negsurf}. 
It seems to be known to experts but we give a proof for the convenience of the reader.  

\section{Proof of theorem}
Fix an algebraically closed field $k$ of characteristic zero.   
Let $Art_k$ be the category of Artinian local $k$-algebras with residue field $k$   
and $Sets$ the category of sets. 
For a proper variety $X$ over $k$ and an effective Cartier divisor $D$ on $X$, let   
$
\Def_{(X,D)} \colon Art_k \rightarrow Sets
$
be the functor sending $A \in Art_k$ to   
 the set of equivalence classes of proper flat morphisms $f \colon X_A \rightarrow \Spec A$ 
together with effective Cartier divisors $D_A \subset X_A$ and marking isomorphisms $\phi_0 \colon X_A \otimes_A k \rightarrow X$ such that 
$\phi_0(D_A \otimes_A k)=D$.  
 This is the pair version of 
the deformation functor $\Def_X$ defined in \cite{kawamataerr}. 
We see that $\Def_{(X,D)}$ is a {\it deformation functor} in the sense of Fantechi--Manetti (\cite[Introduction]{FM}). 

We need the following lemma. 

\begin{lem}\label{base}
Let $Z$ be a smooth proper variety over $k$ and $\Delta \subset Z$ a smooth divisor. 
Set $A_n:= k[t]/(t^{n+1})$ for a non-negative integer $n$. Let $Z_n \rightarrow \Spec A_n$ and $\Delta_n \subset Z_n$ 
be deformations of $Z$ and $\Delta$. Let $\Omega_{Z_n/A_n}^{\bullet} (\log \Delta_n)$ be the de Rham complex of $Z_n/A_n$ with 
logarithmic poles along $\Delta_n$ (cf.\ \cite[(7.1.1)]{IllusieSMF}). 
 Then we have the following:     
\begin{enumerate}
\item[(i)]  the hypercohomology group $\mathbb{H}^{k}(Z_n, \Omega^{\bullet}_{Z_n/A_n}(\log \Delta_n))$ is a free $A_n$-module for all $k$; 
\item[(ii)]the spectral sequence 
\begin{equation}\label{specseq}
E_1^{p,q} := H^q( Z_n, \Omega^p_{Z_n/A_n}(\log \Delta_n)) \Rightarrow \mathbb{H}^{p+q}(Z_n, \Omega^{\bullet}_{Z_n/A_n}(\log \Delta_n))
\end{equation}
degenerates at $E_1$;  
\item[(iii)] 
the cohomology group $H^q( Z_n, \Omega^p_{Z_n/A_n}(\log \Delta_n))$ is a free $A_n$-module and commutes with base change for any $p$ and $q$.  
\end{enumerate}
\end{lem}

\begin{proof}
We can prove this by the same argument as in \cite[Th\'{e}or\`{e}me 5.5]{deligne}. 
We give a proof for the convenience of the reader.
We can assume that $k = \mathbb{C}$ by the Lefschetz principle. 

(i) 
Set $U:= Z \setminus \Delta$. Let $\iota \colon U \hookrightarrow Z$ be the open immersion. 
We see that the complex $\Omega^{\bullet}_{Z_n/A_n}(\log \Delta_n)$ is
quasi-isomorphic to $\iota_* \Omega^{\bullet}_{U_n/A_n}$ by a standard argument as in \cite[Proposition 4.3]{PS}, 
where $U_n \rightarrow \Spec A_n$ is a deformation of $U$ which is induced by $Z_n \rightarrow \Spec A_n$.   
We have an isomorphism
$\mathbb{H}^{k}(Z_n, \iota_* \Omega^{\bullet}_{U_n/A_n}) \simeq \mathbb{H}^{k}(U_n, \Omega^{\bullet}_{U_n/A_n})$ 
since we have $R^i \iota_* \Omega^j_{U_n/A_n} =0$ for $i >0$ and all $j$. 
Moreover we have $ \mathbb{H}^{p+q}(U_n, \Omega^{\bullet}_{U_n/A_n}) \simeq H^{p+q}(U, A_n)$, 
where the latter is the singular cohomology on $U$ with coefficient $A_n$ since 
$\Omega^{\bullet}_{U_n/A_n}$ is a resolution of the sheaf 
$ A_{n, U}$, where $A_{n, U}$ is a constant sheaf on $U$ associated to $A_n$ (See \cite[Lemme 5.3]{deligne}). 
Hence we obtain (i) since we have 
\[
\mathbb{H}^{p+q}(Z_n, \Omega^{\bullet}_{Z_n/A_n}(\log \Delta_n)) \simeq H^{p+q}(U, A_n) \simeq H^{p+q}(U, \mathbb{C}) \otimes A_n.  
\]
Moreover we obtain the equality 
\[
\dim_{\mathbb{C}} \mathbb{H}^{p+q}(Z_n, \Omega^{\bullet}_{Z_n/A_n}(\log \Delta_n))  = \dim_{\mathbb{C}} (A_n) \cdot \dim_{\mathbb{C}} \mathbb{H}^{p+q}(Z, \Omega^{\bullet}_Z(\log \Delta)). 
\]

(ii) By the argument as in \cite[(5.5.5)]{deligne}, we see that 
\begin{equation}\label{1stineq}
\dim_{\mathbb{C}} H^q(Z, \Omega^p_{Z_n/A_n} (\log \Delta_n)) \le \dim_{\mathbb{C}}(A_n) \cdot \dim_{\mathbb{C}} H^q(Z, \Omega^p_Z (\log \Delta))
\end{equation}
and equality holds if and only if $H^q(Z, \Omega^p_{Z_n/A_n}(\log \Delta_n))$ is a free $A_n$-module. 
By the spectral sequence (\ref{specseq}), we have 
\begin{equation}\label{2ndineq}
 \sum_{p+q=k} \dim_{\mathbb{C}} H^q(Z_n, \Omega^p_{Z_n/A_n} (\log \Delta_n)) \ge \dim_{\mathbb{C}} \mathbb{H}^{k} (Z_n, \Omega^{\bullet}_{Z_n/A_n}(\log \Delta_n)).  
 \end{equation}      
 By the two inequalities (\ref{1stineq}), (\ref{2ndineq}) and (i), we obtain 
 \begin{equation}\label{3rdineq}
 \dim_{\mathbb{C}} (A_n) \cdot \sum_{p+q=k} \dim_{\mathbb{C}} H^q(Z, \Omega^p_Z(\log \Delta)) \ge 
\dim_{\mathbb{C}} (A_n) \cdot \dim_{\mathbb{C}} \mathbb{H}^{k} (Z, \Omega^{\bullet}_Z(\log \Delta)).  
 \end{equation}
We have equality in the inequality (\ref{3rdineq}) since  
the spectral sequence (\ref{specseq}) degenerates at $E_1$ when $n=0$ by \cite[Corollaire (3.2.13)(ii)]{Delignehodge2}. 
Hence we have equality in (\ref{2ndineq}) and obtain (ii).  

(iii) This follows from (i) and (ii). 
\end{proof}

To prove Theorem \ref{main}, we prove the following theorem on unobstructedness of deformations of a pair. 

\begin{thm}\label{main1}
Let $X$ be a smooth proper variety such that $H^1(X, \mathcal{O}_X) =0$. 
Assume that there exists a smooth divisor $D \in |{-}mK_X|$ for some positive integer $m$.   
Then 
deformations of $(X,D)$ are unobstructed, that is, $\Def_{(X,D)}$ is a smooth functor. 
\end{thm}

\begin{proof}
Set $A_n:= k[t]/(t^{n+1})$ and $B_n:= k[x,y]/ (x^{n+1},y^2) \simeq A_n \otimes_k A_1$.
For $[ (X_n,D_n), \phi_0] \in \Def_{(X,D)}(A_n)$, let $T^1((X_n,D_n)/A_n)$  be the set of isomorphism classes of pairs 
$((Y_n,E_n), \psi_n) $ consisting of deformations $(Y_n, E_n)$ of $(X_n, D_n)$ over $B_n$ and marking isomorphisms \\ 
$\psi_n \colon Y_n \otimes_{B_n} A_n 
\rightarrow X_n$ such that $\psi_n(E_n \otimes_{B_n} A_n) = D_n$, where we use a ring homomorphism 
$B_n \rightarrow A_n$ given by $x \mapsto t$ and $y \mapsto 0$. 
Then we see the following. 
\begin{claim}We have 
\begin{equation}
T^1((X_n,D_n)/A_n) \simeq H^1(X_n, \Theta_{X_n/A_n}(-\log D_n) ),  
\end{equation} 
where $\Theta_{X_n/A_n}(-\log D_n)$ is the dual of $\Omega^1_{X_n/A_n}(\log D_n)$. 
\end{claim}

\begin{proof}
We can prove this by a standard argument (cf. \cite[Proposition 3.4.17]{Sernesi}) using $B_n = A_n \otimes_k A_1$.  
\end{proof}

Hence, by \cite[Theorem A]{FM}, it is enough to show that the natural homomorphism 
\[
\gamma_n \colon H^1(X_n, \Theta_{X_n/A_n}(-\log D_n)) \rightarrow H^1(X_{n-1}, \Theta_{X_{n-1}/A_{n-1}}(-\log D_{n-1}))
\]
is surjective for the above $X_n, D_n$ and for $X_{n-1}:= X_n \otimes_{A_n} A_{n-1}, D_{n-1}:= D_n \otimes_{A_n} A_{n-1}$. 

Note that we have a perfect pairing 
\[
\Omega^1_{X_n/A_n}(\log D_n) \times \Omega^{d-1}_{X_n/A_n}(\log D_n) \rightarrow \mathcal{O}_{X_n}(K_{X_n/A_n} +D_n) \simeq \omega_{X_n/A_n}^{\otimes 1-m},   
\]
where we set $d:= \dim X$. 
We have $ \mathcal{O}_{X_n}(K_{X_n/A_n} +D_n) \simeq \omega_{X_n/A_n}^{\otimes 1-m}$ since we have $H^1(X, \mathcal{O}_X) =0$ 
(See \cite[Theorem 6.4(b)]{Hdef}, for example.). Thus we see that 
\[
H^1(X_n, \Theta_{X_n/A_n}(-\log D_n)) \simeq H^1(X_n, \Omega^{d-1}_{X_n/A_n}(\log D_n) \otimes \omega_{X_n/A_n}^{\otimes m-1}).  
\]

Let 
\[
\pi_n \colon Z_n:= \Spec \bigoplus_{i=0}^{m-1} \mathcal{O}_{X_n}(i K_{X_n/A_n}) \rightarrow X_n 
\]
be the ramified covering defined by a section $\sigma_{D_n} \in H^0(X_n, -m K_{X_n/A_n})$ which corresponds to $D_n$. We have an isomorphism 
\[
\pi_n^* \Omega^1_{X_n/A_n}(\log D_n) \simeq \Omega^1_{Z_n/A_n}(\log \Delta_n)
\]
for some divisor $\Delta_n \in |{-}\pi_n^* K_{X_n/A_n}|$. Hence we see that 
\[
(\pi_n)_* \Omega^{d-1}_{Z_n/A_n} (\log \Delta_n) \simeq \bigoplus_{i=0}^{m-1} \Omega^{d-1}_{X_n/A_n}(\log D_n) (i K_{X_n/A_n}) 
\]
and $\Omega^{d-1}_{X_n/A_n}(\log D_n) \otimes \omega_{X_n/A_n}^{\otimes m-1}$ is one of the direct summands. 

Hence it is enough to show that the natural restriction homomorphism 
\[
r_n \colon H^1(Z_n, \Omega^{d-1}_{Z_n/A_n}(\log \Delta_n)) \rightarrow H^1(Z_{n-1}, \Omega^{d-1}_{Z_{n-1}/A_{n-1}}(\log \Delta_{n-1})) 
\]
is surjective, where we set $Z_{n-1}:= Z_n \otimes_{A_n} A_{n-1}$ and $\Delta_{n-1}:= \Delta_n \otimes_{A_n} A_{n-1}$, since $\gamma_n$ is an eigenpart of $r_n$. 
By Lemma \ref{base}(iii), we see the required surjectivity. This completes the proof of Theorem \ref{main1}. 
\end{proof}

\begin{rem}
Iacono \cite{iacono} proved Theorem \ref{main1} when $m=1$ without the assumption $H^1(X, \mathcal{O}_X)=0$ in \cite[Corollary 4.5]{iacono} 
as a consequence of the analysis of DGLA. 
\end{rem}

\begin{rem}
We can remove the assumption $H^1(X,\mathcal{O}_X)=0$ when $m=1$ by a similar argument as in \cite[Corollary 2]{ran}. 
In that case, we see that $\mathcal{O}_{X_n}(K_{X_n/A_n} + D_n) \simeq \mathcal{O}_{X_n}$ since we have $H^0(X_n, K_{X_n/A_n}+D_n) \simeq A_n$ 
by Claim \ref{base}, with $X_n, D_n$ as in the proof of Theorem \ref{main1}. 

We do not know whether we can remove the assumption $H^1(X, \mathcal{O}_X)=0$ in Theorem \ref{main1} when $m$ is arbitrary. 
\end{rem}

\vspace{5mm}

Theorem \ref{main1} implies Theorem \ref{main} as follows. 

\begin{proof}[Proof of Theorem \ref{main}]
Since $H^1(D, \mathcal{N}_{D/X})=0$, we see that the forgetful morphism 
\[
\Def_{(X,D)} \rightarrow \Def_X
\]
 between functors is smooth. 
Since $\Def_{(X,D)}$ is smooth by Theorem \ref{main1}, we see that $\Def_X$ is also smooth. 
\end{proof}

\vspace{5mm}

Theorem \ref{main} implies Theorem \ref{wfanothm} as follows. 

\begin{proof}[Proof of Theorem \ref{wfanothm}]
Let $X$ be a weak Fano manifold of dimension $d$. 
By the base point free theorem, we can take a sufficiently large integer $m$  such that $-mK_X$ is base point free and 
contains a smooth element $D \in |{-}mK_X|$.  
We have $H^1(D, \mathcal{N}_{D/X})=0$  since there is an exact sequence 
\[
H^1(X, \mathcal{O}_X (D)) \rightarrow H^1(D, \mathcal{N}_{D/X}) \rightarrow H^2(X,\mathcal{O}_X) 
\] 
and both outer terms are zero by the Kawamata--Viehweg vanishing theorem. 
Hence Theorem \ref{main} implies Theorem \ref{wfanothm}. 
\end{proof}

\begin{rem}
We can prove the following theorem by the same argument as Theorem \ref{wfanothm}. 

\begin{thm}\label{moishezon}
Let $X$ be a complex manifold whose anticanonical bundle is nef and big. 
Then deformations of $X$ are unobstructed. 
\end{thm}
Actually we see that such a complex manifold is Moishezon since there is a big divisor on $X$. 
Hence we can show Lemma \ref{base} and the base-point free theorem in this setting. 
Using these, we can show Theorem \ref{moishezon} in the same way as Theorem \ref{wfanothm}.  
\end{rem}

\begin{eg}
We give an example of a weak Fano manifold such that $H^2(X, \Theta_X) \neq 0$, where $\Theta_X$ is the tangent sheaf. 

Let $f \colon X \rightarrow \mathbb{P}(1,1,1,3)$ be the blow-up of the singular point $p$ of the weighted projective space. 
We can check that $X \simeq \mathbb{P}_{\mathbb{P}^2}(\mathcal{O}_{\mathbb{P}^2} \oplus \mathcal{O}_{\mathbb{P}^2}(-3))$
 and $f$ is the anticanonical morphism of $X$. 
Hence $-K_X = f^* (- K_{\mathbb{P}(1,1,1,3)})$ and this is nef and big. 
Set $\mathcal{E}:= \mathcal{O}_{\mathbb{P}^2} \oplus \mathcal{O}_{\mathbb{P}^2}(-3)$. 
By a direct calculation using the relative Euler sequence for $\mathbb{P}_{\mathbb{P}^2}(\mathcal{E}) \rightarrow \mathbb{P}^2$, 
we see that 
\[
h^2(X, \Theta_X) = h^2(X, \Theta_{X/ \mathbb{P}^2}) = h^2(\mathbb{P}^2, \mathcal{E} \otimes \mathcal{E}^*) = 1. 
\]
Hence $H^2(X, \Theta_X) \neq 0$. 

Thus we need a technique such as $T^1$-lifting for the proof of Theorem \ref{main}. 
\end{eg}

\begin{eg}\label{egelep}
We give an example of a Fano manifold such that neither of the linear systems $|{-}K_X|$ and $|{-}2K_X|$ contain smooth elements. 
Our example is a modification of an example in \cite[Example 3.2 (3)]{effnv}.

Let $X:= X_{5d} \subset \mathbb{P}(1,\ldots,1,5,d) = \mathbb{P}(1^n,5,d)$ be a weighted hypersurface of degree $5d$ and dimension $n$. 
Assume that $d \not\equiv 0  \mod 5$ and that $5+n-4d =2$. (For example, $d=6, n=21$.) 
The latter condition implies that $-K_X = \mathcal{O}_X(2)$. 
We see that the base locus of $|{-}K_X|$ and $|{-}2K_X|$ consists of a point $p:= H_1 \cap \ldots \cap H_n \cap X_{5d}$, where 
$H_1, \ldots, H_n$ are degree $1$ hyperplanes of the first $n$ coordinates of $\mathbb{P}(1^n,5,d)$. 
We see that every element of $|{-}K_X|$ has multiplicity $2$ at the base point $p$ and hence is singular. 
We also see that  every element of $|{-}2K_X|$ has multiplicity $4$ at the base point $p$ and hence is singular. 
\end{eg}

\begin{eg}
We give an example of a smooth projective variety such that $\Def_X$ is not smooth and $-K_X$ is big. 

Let $C \subset \mathbb{P}^3$ be a smooth curve with an obstructed embedded deformation which lies in a cubic surface as in \cite[Theorem 13.1]{Hdef}. 
Let $\mu \colon X \rightarrow \mathbb{P}^3$ be the blow-up of $\mathbb{P}^3$ along $C$. Then $X$ has an obstructed deformation. 
See \cite[Example 13.1.1]{Hdef}. 
Note that 
${-}K_X = \mu^* \mathcal{O}_{\mathbb{P}^3}(4)- E$ where $E:= \mu^{-1}(C)$ and $C$ is contained in a cubic surface $S \subset \mathbb{P}^3$. 
Let $\tilde{S} \subset X$ be the strict transform of $S$. 
Then we see that $-K_X$ is big since $\tilde{S} + |\mu^* \mathcal{O}_{\mathbb{P}^3}(1)| \subset |{-}K_X| $.  
\end{eg}

\begin{eg}
We give an example of $X$ and $D \in |{-}K_X|$ such that $\Def_{(X,D)}$ is smooth but $\Def_X$ is not smooth. 

Let $C \subset \mathbb{P}^3$ be a smooth curve in a quartic surface $S$ such that the Hilbert scheme of curves in $\mathbb{P}^3$ is singular 
at the point corresponding to $C$ (cf. \cite[Exercise 13.2]{Hdef}). 
Let $X \rightarrow \mathbb{P}^3$ be the blow-up of $\mathbb{P}^3$ along $C$. Then $X$ has an obstructed deformation. 
However the strict transform $D:= \tilde{S} \in |{-}K_X|$ of $S$ is smooth and $H^1(X, \mathcal{O}_X)=0$. Hence $\Def_{(X,D)}$ is smooth by Theorem \ref{main1}. 
\end{eg}

\begin{eg}
We give an example of $X$ with an obstructed deformation such that $-K_X$ is nef. 

Set $X := T^m \times \mathbb{P}^1$ where $T^m$ is a complex torus of dimension $m \ge 2$. Then $X$ has an obstructed deformation (\cite[p.436--441]{kodaira}). 
Note that $-K_X$ is nef. It is actually semiample.  
\end{eg}

It is natural to ask the following question:  

\begin{prob}
Let $X$ be a smooth projective variety such that $-K_X$ is nef and $H^1(X,\mathcal{O}_X)=0$. Is the Kuranishi space of $X$ smooth? 
\end{prob}

\section{The surface case}

The following lemma states that smoothness of the Kuranishi space is preserved under the blow-up at a point. 

\begin{lem}\label{minred}
Let $S$ be a smooth projective variety and $\nu \colon T \rightarrow S$ the blow-up at a point $p \in S$. 
Then the functor $\Def_S$ is smooth if and only if the functor $\Def_{T}$ is smooth. 
\end{lem}

\begin{proof}
Let $\Def_{(S,p)}$ be the functor of deformations of a closed immersion $\{p \} \subset S$ and  
$\Def_{(T,E)}$ the functor as in Section 2, where $E:= \nu^{-1}(p)$.
We can define a natural transformation  
\[
\nu_* \colon \Def_{(T,E)} \rightarrow \Def_{(S,p)}  
\]
as follows:  
given $A \in Art_k$ and a deformation $(\mathbf{T}, \mathbf{E})$ of $(T,E)$ over $A$,  
we see that $\nu_* \mathcal{O}_{\mathbf{T}}$ is a sheaf of flat $A$-algebras by \cite[Corollary 0.4.4]{wahl} since we have $R^1\nu_* \mathcal{O}_T =0$.   
We also see that $\nu_* \mathcal{O}_{\mathbf{T}}(- \mathbf{E})$ is a sheaf of flat $A$-modules 
by \cite[Corollary 0.4.4]{wahl} since we have $R^1\nu_* \mathcal{O}_T (-E) =0$ by a direct calculation. 
Hence we can define a deformation $(\mathbf{S}, \mathbf{p})$ of $(S,p)$ over $A$ by sheaves $\mathcal{O}_{\mathbf{S}}:= \nu_* \mathcal{O}_{\mathbf{T}}, 
\mathcal{I}_{\mathbf{p}}:=  \nu_* \mathcal{O}_{\mathbf{T}}(- \mathbf{E})$ and obtain a natural transformation $\nu_*$. 

We can also define a natural transformation 
\[
\nu^* \colon \Def_{(S,p)} \rightarrow \Def_{(T,E)} 
\]
as follows: given a deformation $(\mathbf{S}, \mathbf{p})$ of $(S,p)$ over $A \in Art_k$, we define a deformation $\mathbf{T}$ of $T$ 
as the blow-up of $\mathbf{S}$ along $\mathbf{p}$. 
We can also define a deformation $\mathbf{E}$ of $E$ by the inverse image ideal sheaf $\nu^{-1} \mathcal{I}_{\mathbf{p}} \cdot \mathcal{O}_{\mathbf{T}}$, 
where $\mathcal{I}_{\mathbf{p}}$ is the ideal sheaf of $\mathbf{p} \subset \mathbf{S}$. 

We see that $\nu_*$ and $\nu^*$ are inverse to each other. Hence we have $\Def_{(T,E)} \simeq \Def_{(S,p)}$ as functors. 

We have forgetful morphisms of functors $F_T \colon \Def_{(T,E)} \rightarrow \Def_T$ and $F_S \colon \Def_{(S,p)} \rightarrow \Def_S$. 
We see that $F_T$ and $F_S$ are smooth since we have $H^1(E, \mathcal{N}_{E/T}) \simeq H^1(\mathbb{P}^{d-1}, \mathcal{O}_{\mathbb{P}^{d-1}}(-1)) =0$ and $H^1(\mathcal{N}_{p/S})=0$, where we set $d:= \dim S$. 

Thus we have a diagram 
\[
\xymatrix{
\Def_{(T,E)} \ar[r]^{\sim} \ar[d]^{F_T} & \Def_{(S,p)} \ar[d]^{F_S} \\
\Def_{T} & \Def_S,  
}
\]
where $F_T$ and $F_S$ are smooth. 
Hence we see the required equivalence. 
\end{proof}

By this lemma, we see that a smooth projective surface has unobstructed deformations if and only if 
its relatively minimal model has unobstructed deformations. 

Using Lemma \ref{minred}, we can prove the following:  
\begin{thm}\label{negsurf}
Let $X$ be a smooth projective surface with non-positive Kodaira dimension. 
Then the deformations of $X$ are unobstructed. 
\end{thm}

\begin{proof}
By Lemma \ref{minred}, we can assume that $X$ does not contain a $-1$-curve. 

If the Kodaira dimension of $X$ is negative, it is known that $X \simeq \mathbb{P}^2$ or $X \simeq \mathbb{P}_C(\mathcal{E})$ for some projective curve $C$ and a rank $2$ vector bundle 
$\mathcal{E}$ on $C$. 
In these cases, we see that $H^2(X, \Theta_X) = 0$ by the Euler sequence or the argument  in \cite[p.204]{seiler}. 

If the Kodaira dimension of $X$ is zero, it is a $K3$ surface, an Abelian surface or its \'{e}tale quotient. 
It is well known that these surfaces have unobstructed deformations. Hence we are done. 
\end{proof}

\begin{rem}
Kas \cite{kas} gave an example of a smooth projective surface of Kodaira dimension $1$ with an obstructed deformation.  
\end{rem}

\section*{Acknowledgments}
The author would like to thank Professors Osamu Fujino, Yoshinori Gongyo, Andreas H\"{o}ring, Donatella Iacono, Yujiro Kawamata, Marco Manetti, Tatsuhiro Minagawa, Yoshinori Namikawa and Hirokazu Nasu for valuable conversations. 
He thanks the referee for constructive comments. 
He thanks Professor Miles Reid and Michael Selig for improving the presentation of the manuscript and valuable conversations. 
He is partially supported by Warwick Postgraduate Research Scholarship.

\end{document}